\documentclass[10pt,bezier]{article}
\usepackage{amsmath,amssymb,amsfonts,graphicx,caption,subcaption}
\usepackage[colorlinks,linkcolor=red,citecolor=blue]{hyperref}

\newtheorem{prethm}{{\bf Theorem}}

\newenvironment{thm}{\begin{prethm}{\hspace{-0.5
em}{\bf.}}}{\end{prethm}}
\newtheorem{prepro}{{\bf Theorem}}

\newtheorem{precor}{{\bf Corollary}}

\newenvironment{cor}{\begin{precor}{\hspace{-0.5
em}{\bf.}}}{\end{precor}}
\newtheorem{preconj}{{\bf Conjecture}}

\newenvironment{conj}{\begin{preconj}{\hspace{-0.5
em}{\bf.}}}{\end{preconj}}
\newtheorem{preremark}{{\bf Remark}}

\newtheorem{prelem}{{\bf Lemma}}

\newtheorem{preproof}{{\bf Proof.}}

\newenvironment{proof}[1]{\begin{preproof}{\rm
#1}\hfill{$\Box$}}{\end{preproof}}

\title{\bf\Large 
 Tutte's $3$-Flow Conjecture in $3$-tree-connected graphs
}
\author{
{\normalsize{\sc Morteza Hasanvand${}$},
}\vspace{3mm}
\\{\footnotesize{${}$\it Department of Mathematical
 Sciences, Sharif
University of Technology, Tehran, Iran}}
{\footnotesize{}}\\{\footnotesize{ $\mathsf{morteza.hasanvand@alum.sharif.edu }$ }}}

\date{}
\begin{document}
\maketitle
\begin{abstract}{
Tutte's $3$-flow conjecture says that every $4$-edge-connected graph admits a nowhere-zero $3$-flow. Kochol (2001) showed that it is enough to prove this conjecture for $5$-edge-connected graphs. Former, Jaeger, Linial, Payan, and Tarsi (1992) conjectured that every $5$-edge-connected graph is $Z_3$-connected and so it admits a nowhere-zero $3$-flow. In this note, we show that if the second conjecture would be true, then every $3$-tree-connected graph must also be $Z_3$-connected and so Tutte's $3$-flow conjecture can be extended to this family of graphs.
} {\small
}
\end{abstract}
%
%
%
%
%
%
%
%
%
%
%
%
\section*{Introduction}
In this note, all graphs have no loop, but multiple edges are allowed.
Let $G$ be a graph, let $k$ be an integer, $k\ge 2$, and let $p:V(G)\rightarrow Z_k$ be a mapping such that $|E(G)|\stackrel{k}{\equiv} \sum_{v\in V(G)}p(v)$, where $Z_k$ is the cyclic group of order $k$. 
An orientation of $G$ is called {\it $p$-orientation}, if for every vertex $v$, $d_G^+(v)\stackrel{k}{\equiv}p(v) $, where $d_G^+(v)$ denotes the out-degree of $v$. 
A graph $G$ is called {\it $Z_3$-connected}, if it admits a $p$-orientation, for every mapping $p:V(G)\rightarrow Z_3$ 
satisfying $|E(G)|\stackrel{3}{\equiv} \sum_{v\in V(G)}p(v)$.
We say that a graph $G$ admits a noweherezero $3$-flow, if it admits a {\it $p$-orientation} in which each vertex $v$, $2p(v)\stackrel{3}{\equiv}d_G(v)$. According these definitions, if $G$ is $Z_3$-connected, then obviously it must admit a noweherezero $3$-flow. Note that these definitions are equivalent to their initial well-known definitions, see~\cite{Lai-2007, Lovasz-Thomassen-Wu-Zhang-2013}.
A graph $G$ is called $k$-tree-connected, if it contains $k$ edge-disjoint spanning trees. Note that every $2k$-edge-connected graph is also $k$-tree-connected~\cite{Nash-Williams-1961,Tutte-1961}.

Tutte's $3$-Flow Conjecture says that every $4$-edge-connected graph admits a nowhere-zero $3$-flow.
Jaeger, Linial, Payan, and Tarsi (1992) proposed a stronger conjecture which says every $5$-edge-connected graph is $Z_3$-connected. In 2001 Kochol~\cite{Kochol-2001} proved that if every $5$-edge-connected graph admits a nowhere-zero $3$-flow, then Tutte's $3$-Flow Conjecture is true.
\begin{conj}{\rm (\cite{Jaeger-Linial-Payan-Tarsi-1992})}\label{intro:conj:5-edge-connected}
{Let $G$ be a graph and let $p:V(G)\rightarrow Z_3$ be a mapping such that $|E(G)|\stackrel{3}{\equiv}\sum_{v\in V(G)}p(v)$.
 If $G$ is $5$-edge-connected, then it admits a $p$-orientation.
}\end{conj}

In 2012 Thomassen~\cite{Thomassen-2012} succeeded to confirm Conjecture~\ref{intro:conj:5-edge-connected} for $8$-edge-connected graph. Later, Lov{\'a}sz, Thomassen, Wu, and Zhang (2013)~\cite{Lovasz-Thomassen-Wu-Zhang-2013} improved Thomassen's result to the following version by pushing down the needed edge-connectivity by one. 
\begin{thm}{\rm (\cite{Lovasz-Thomassen-Wu-Zhang-2013})}\label{intro:thm:Lovase, Thomassen, Wu, Zhang}
{Let $G$ be a graph and let $p:V(G)\rightarrow Z_3$ be a mapping such that $|E(G)|\stackrel{3}{\equiv}\sum_{v\in V(G)}p(v)$.
 If $G$ is $6$-edge-connected, then it admits a $p$-orientation.
}\end{thm}

In~\cite{Han-Lai-Li-2018} the authors used a stronger version of their result to confirm Conjecture~\ref{intro:conj:5-edge-connected} for $4$-tree-connected graphs. In this note, we show that if Conjecture~\ref{intro:conj:5-edge-connected} would be true, then that conjecture together with Tutte's $3$-Flow Conjecture can be developed to $3$-tree-connected graphs. 
Note that this number is sharp, because the complete graph of order $4$ does not have as a nowhere-zero $3$-flow, while it is $2$-tree-connected. 
\begin{thm}{\rm (\cite{Han-Lai-Li-2018})}
{Let $G$ be a graph and let $p:V(G)\rightarrow Z_3$ be a mapping such that $|E(G)|\stackrel{3}{\equiv}\sum_{v\in V(G)}p(v)$.
 If $G$ is $4$-tree-connected, then it admits a $p$-orientation.
}\end{thm}
%
%
%
%
%
%
%
%
%
%
%
%
%
\section{Orientations modulo $3$ in $3$-tree-connected graphs}
%
%
The following theorem shows a consequence of Conjecture~\ref{intro:conj:5-edge-connected}.

\begin{thm}\label{thm:main}
{Assume that Conjecture~\ref{intro:conj:5-edge-connected} is true. Let $G$ be a graph, let $p:V(G)\rightarrow Z_3$ be a mapping such that $|E(G)|\stackrel{3}{\equiv}\sum_{v\in V(G)}p(v)$. 
If $G$ is $3$-tree-connected, then
it admits a $p$-orientation.
}\end{thm}
\begin{proof}
{Let $k=3$ and $\lambda=5$. By induction on $|V(G)|$. For $|V(G)|=1$, the proof is clear. 
So, suppose $|V(G)|\ge 2$. 
If $G$ is $\lambda$-edge-connected, then it follows from the assumption.
Thus, we assume $G$ is not $\lambda$-edge-connected.
 Let $T_1\ldots, T_{\lambda-2}$ be $\lambda-2$ edge-disjoint spanning trees of $G$.
Let $\mathcal{C}$ be an edge cut of $G$ with the minimum size.
Note that $|\mathcal{C}|\le \lambda-1$ and $G\setminus \mathcal{C}$ is composed by two disjoint connected graphs $G_1$ and $G_2$.
Let $r\in \{0,\ldots, k-1\}$ be the unique integer with 
$$r+|E(G_1)|\stackrel{k}{\equiv} \sum_{v\in V(G_1)}p(v).$$
Suppose first that $|E(\mathcal{C})|=\lambda-2$. 
This implies that for each tree $T_j$, $|E(\mathcal{C})\cap E(T_j)|=1$, and hence every graph $G_j$ contains $\lambda-2$ edge-disjoint spanning trees. 
Since $|E(\mathcal{C})|=\lambda-2\ge k-1\ge r$, one can orient $r$ edges of $\mathcal{C}$ from $G_1$ to $G_2$ and $|E(\mathcal{C})|-r$ remaining edges from $G_2$ to $G_1$.
Now, for every graph $G_j$,
let $p_j:V(G_j)\rightarrow Z_k$ be a mapping such that for each vertex 
$v\in V(G_j)$, $p_j(v)=p(v)-q_j(v)$, where $q_j(v)$ is the number of edges of $\mathcal{C}$ directed away from $v$. 
Since $r=\sum_{v\in V(G_1)}q_1(v)$, we have $|E(G_1)|\stackrel{k}{\equiv}\sum_{v\in V(G_1)}p_1(v)$.
Also, since $|\mathcal{C}|-r=\sum_{v\in V(G_2)}q_2(v)$ and $|E(G)|\stackrel{k}{\equiv}\sum_{v\in V(G)}p(v)$, we have $|E(G_2)|\stackrel{k}{\equiv}\sum_{v\in V(G_2)}p_2(v)$. Thus the induction hypothesis implies that every graph $G_j$ has a $p_j$-orientation. It is not hard to see that these orientations of $G_1$, $\mathcal{C}$, and $G_2$ induce a $p$-orientation for $G$.
 In future cases, we leave some details for the reader in order to apply the induction hypothesis.

Now, suppose that $|E(\mathcal{C})|=\lambda-1$. 
Without loss of generality assume that $|E(\mathcal{C})\cap E(T_1)|=2$ and $|E(\mathcal{C})\cap E(T_j)|=1$ for each tree $T_j$ with $j>1$. 
Also, without loss of generality assume that the spanning graph of $G_1$ with the edge set $E(G_1)\cap E(T_1)$ is connected. Therefore, $G_1$ contains $\lambda-2$ edge-disjoint spanning trees.
 In $G$, contract all vertices of $G_1$ to a single vertex (by removing loops) and call the resulting graph $G'_2$. It is easy to see that $G'_2$ contains $\lambda-2$ edge-disjoint spanning trees. 
Let $p'_2:V(G'_2)\rightarrow Z_k$ be a mapping such that for each vertex $v\in V(G_2')\cap V(G_2)$, $p'_2(v)=p(v)$ and for the vertex $u\in V(G_2')$ corresponding to $G_1$, $p'_2(u)=r$. 

Now, if $|V(G_1)|> 1$ then by the induction hypothesis, the graph $G_2'$ has a $p'_2$-orientation.
This orientation of $G_2'$ induces an orientation for $\mathcal{C}$.
In this case, let $p_1:V(G_1)\rightarrow Z_k$ be a mapping such that for each vertex $v\in V(G_1)$, $p_1(v)=p(v)-q(v)$, where $q(v)$ is the number of edges of $ \mathcal{C}$ directed away from $v$. 
At present, by the induction hypothesis, the graph $G_1$ has a $p_1$-orientation. It is not hard to see that the $p_1$-orientation of $G_1$ and $p'_2$-orientation of $G'_2$ induce a $p$-orientation for~$G$. 

In the final case, set $V(G_1)=\{u\}$. Let $xu$ and $uy$ be the two edges of $T_1$ incident to $u$.
 Now, remove the vertex $u$ from $G$ and add a new edge $xy$ to $G$.
 Call the resulting graph $H$. 
Since $T_1$ is a tree with no multiple edges, $x\neq y$ and so $H$ has no loop. Note also that $|E(H)|=|E(G)|-|E(\mathcal{C})|+1$.
It is not hard to see that $H$ contains $\lambda-2$ edge-disjoint spanning trees.
In $G$ orient the edge $xu$ from $x$ to $u$ and orient the edge $uy$ from $u$ to $y$. 
Next, orient exactly $r_u-1$ edges of $E(\mathcal{C})\setminus \{xu,uy\}$ 
away from $u$ and orient all remaining edges of $E(\mathcal{C})$ in opposite direction, 
where $r_u\in \{1,\ldots, k\}$ is the unique integer with
$r_u\stackrel{k}{\equiv}p(u).$
Since $|E(\mathcal{C})\setminus \{xu, uy\}|=\lambda-3\ge k-1\ge r_u-1$, this orientation of $\mathcal{C}$ is possible. 
Now, let $p_0:V(H)\rightarrow Z_k$ be a mapping such that for each vertex $v\in V(H)$, $p_0(v)=p(v)-q(v)$, where $q(v)$ is the number of edges of $\mathcal{C}\setminus \{xu,uy\}$ directed away from $v$.
By the induction hypothesis, $H$ has a $p_0$-orientation. 
It is not difficult to see that the $p_0$-orientation of $H$ and the orientation of $\mathcal{C}$ induce a $p$-orientation for $G$. 
Notice that if the edge $xy$ of $H$ oriented from $y$ to $x$, for inducing, the direction of two edges $xu$ and $uy$ in $G$ must be reversed. This completes the proof.
}\end{proof}
By combining Theorem 2.1 in \cite{Equitable} with Theorem~\ref{thm:main}, one can derive the following corollary.
\begin{cor}
{Assume that Conjecture~\ref{intro:conj:5-edge-connected} is true. 
Let $G$ be a $3$-tree-connected graph. Then 
 $G$ does not have exactly one vertex $z$ satisfying $d_G(z) \stackrel{3}{\not\equiv}0$ 
if and only if it can be edge-decomposed into three factors $G_1$, $G_2$, and $G_3$ such that for each $v\in V(G_i)$, $|d_{G_i}(v)-d_G(v)/3|< 1$.
}\end{cor}
%
%
%
%
\section{Conclusion: A generalization}
In 2006 Bar{\'a}t and Thomassen~\cite{Barat-Thomassen-2006} conjectured that for every tree $T$ there exists a natural number $k_T$ such that every $k_T$-edge-connected simple graph of size divisible by $|E(T)|$ has a $T$-edge-decomposition. 
However, this conjecture investigates only the existence of $k_T$, an upper bound on $k_T$ was stated in~\cite{Barat-Gerbner-2014} as the following conjecture. 
A consequence of Theorem~\ref{thm:main} says that the following conjecture is true, if Conjecture~\ref{intro:conj:5-edge-connected} would be true (using the special case $p=0$).
\begin{conj}{\rm(\cite{Barat-Gerbner-2014})}\label{conj:star}
{Let $G$ be a simple graph of size divisible by $k$ with $k\ge 1$, and let $T$ be a tree of size $k$. If $G$ is $k$-tree-connected, then it admits a $T$-edge-decomposition.
}\end{conj}
When $k\ge 3$, the special case $k$-star of the above-mentioned conjecture can conclude the next conjecture (more precisely they are equivalent), using an idea that was used in~\cite{Lai-2007}. To see this, for every vertex $v$ of the graph $G$, replace a large graph $H_v$ containing $k$ edge-disjoint spanning trees with $|E(H_v)|+p(v) \stackrel{k}{\equiv}0$ such that after replacing the resulting graph forms a simple graph.
Since the new graph contains $k$ edge-disjoint spanning trees and its size is divisible by $k$, by the assumption it admits a $k$-star-decomposition. Now, orient the edges of these stars away from their centres. 
It is not difficult to see that this orientation induces a $p$-orientation for~$G$.
\begin{conj}\label{conj:f}
{Let $G$ be a graph, let $k$ be an integer, $k\ge 3$, and let $p:V(G)\rightarrow Z_k$ be a mapping such that $|E(G)|\stackrel{k}{\equiv}\sum_{v\in V(G)}p(v)$. If $G$ is $k$-tree-connected, then
it has a $p$-orientation.
}\end{conj}
%
%
 
%


\begin{thebibliography}{10}

\bibitem{Barat-Gerbner-2014}
J.~Bar{\'a}t and D.~Gerbner, Edge-decomposition of graphs into copies of a tree with four edges, Electron. J. Combin., 21 (2014),~Paper 1.55, 11.

\bibitem{Barat-Thomassen-2006}
J.~Bar{\'a}t and C.~Thomassen, Claw-decompositions and {T}utte-orientations, J. Graph Theory 52 (2006)~135--146.

\bibitem{Han-Lai-Li-2018}
 Han, Lai, and Li, Nowhere-zero $3$-flow and $Z_3$-connectedness in graphs with four edge-disjoint spanning trees, J. Graph Theory 88 (2018) 577--591.

\bibitem{Equitable}
M.~Hasanvand, Equitable factorizations of edge-connected graphs, arXiv:1906.04325.

\bibitem{Jaeger-Linial-Payan-Tarsi-1992}
F. Jaeger, N. Linial, C. Payan, and M. Tarsi, Group connectivity of graphs---a nonhomogeneous analogue of nowhere-zero flow properties, J. Combin. Theory Ser. B 56 (1992)~165--182.

\bibitem{Kochol-2001}
M.~Kochol, An equivalent version of the 3-flow conjecture, J. Combin. Theory Ser. B 83 (2001)~258--261.

\bibitem{Lai-2007}
 L.~Lai, Mod {$(2p+1)$}-orientations and {$K_{1,2p+1}$}-decompositions, SIAM J. Discrete Math. 21 (2007)~844--850.

\bibitem{Lovasz-Thomassen-Wu-Zhang-2013}
 L.M. Lov{\'a}sz, C.~Thomassen, Y.~Wu, and C.-Q. Zhang, Nowhere-zero 3-flows and modulo {$k$}-orientations, J. Combin. Theory Ser. B 103 (2013)~587--598.

\bibitem{Nash-Williams-1961}
C.St.J.A. Nash-Williams, Edge-disjoint spanning trees of finite graphs, J. London Math. Soc. 36 (1961)~445--450.

\bibitem{Thomassen-2012}
C.~Thomassen, The weak 3-flow conjecture and the weak circular flow conjecture, J. Combin. Theory Ser. B 102 (2012)~521--529.

\bibitem{Tutte-1961}
 W.T. Tutte, On the problem of decomposing a graph into {$n$} connected factors, J. London Math. Soc. 36 (1961)~221--230.

\end{thebibliography}
\end{document}